\documentclass[12pt,a4paper]{article}%
\usepackage[utf8]{inputenc}
\usepackage{amsmath}
\usepackage{amsfonts}
\usepackage{amssymb}
\usepackage{graphicx}
\usepackage[space]{grffile}%
\providecommand{\U}[1]{\protect\rule{.1in}{.1in}}
\newtheorem{theorem}{Theorem}

\newtheorem{conjecture}[theorem]{Conjecture}
\newtheorem{corollary}[theorem]{Corollary}

\newtheorem{lemma}[theorem]{Lemma}

\newtheorem{proposition}[theorem]{Proposition}

\newenvironment{proof}[1][Proof]{\noindent\textbf{#1.} }{\ \hfill \rule{0.5em}{0.5em}\bigskip}
\graphicspath{{D:/Dropbox/Riste-Sedlar/MetricDim/Paper 7/Slike/}}

\begin{document}

\title{Remarks on the vertex and the edge metric dimension of $2$-connected graphs}
\author{Martin Knor$^{1}$,
\and Jelena Sedlar$^{2,4}$,\\Riste \v{S}krekovski$^{3,4}$ \\[0.3cm] {\small $^{1}$ \textit{Slovak University of Technology in Bratislava,
Bratislava, Slovakia}}\\[0.3cm] {\small $^{2}$ \textit{University of Split, Faculty of civil
engineering, architecture and geodesy, Croatia}}\\[0.1cm] {\small $^{3}$ \textit{University of Ljubljana, FMF, 1000 Ljubljana,
Slovenia }}\\[0.1cm] {\small $^{4}$ \textit{Faculty of Information Studies, 8000 Novo
Mesto, Slovenia }}\\[0.1cm] }
\maketitle

\begin{abstract}
The vertex (resp. edge) metric dimension of a graph $G$ is the size of a
smallest vertex set in $G$ which distinguishes all pairs of vertices (resp.
edges) in $G$ and it is denoted by $\mathrm{dim}(G)$ (resp. $\mathrm{edim}%
(G)$). The upper bounds $\mathrm{dim}(G)\leq2c(G)-1$ and $\mathrm{edim}%
(G)\leq2c(G)-1,$ where $c(G)$ denotes the cyclomatic number of $G$, were
established to hold for cacti without leaves distinct from cycles, and
moreover all leafless cacti which attain the bounds were characterized. It was
further conjectured that the same bounds hold for general connected graphs
without leaves and this conjecture was supported by showing that the problem
reduces to $2$-connected graphs. In this paper we focus on $\Theta$-graphs, as
the most simple $2$-connected graphs distinct from cycle, and show that the
the upper bound $2c(G)-1$ holds for both metric dimensions of $\Theta$-graphs
and we characterize all $\Theta$-graphs for which the bound is attained. We
conclude by conjecturing that there are no other extremal graphs for the bound
$2c(G)-1$ in the class of leafless graphs besides already known extremal cacti
and extremal $\Theta$-graphs mentioned here.

\end{abstract}

\section{Introduction}

In this paper we assume that all graphs are simple and connected, unless we
explicitly say otherwise, and we consider distances in such graphs. Let $G$ be
a graph with the set of vertices $V(G)$ and the set of edges $E(G).$ The
\emph{distance} $d_{G}(u,v)$ between vertices $u,v\in V(G)$ is the length of a
shortest path in $G$ connecting vertices $u$ and $v.$ The distance
$d_{G}(u,e)$ between a vertex $u\in V(G)$ and an edge $e=vw\in E(G)$ is
defined by $d_{G}(u,e)=\min\{d_{G}(u,v),d_{G}(u,w)\}.$ When no confusion
arises from that, we use abbreviated notation $d(u,v)$ and $d(u,e).$ We say
that a pair $x$ and $x^{\prime}$ of vertices from $V(G)$ (resp. of edges from
$E(G)$) is \emph{distinguished} by a vertex $s\in V(G)$ if $d(s,x)\not =%
d(s,x^{\prime})$. A set $S$ is a \emph{vertex} (resp. an \emph{edge})
\emph{metric generator} if every pair $x$ and $x^{\prime}$ of vertices from
$V(G)$ (resp. of edges from $E(G)$) is distinguished by a vertex $s\in S.$ The
size of a smallest vertex (resp. edge) metric generator in $G$ is called the
vertex (resp. the edge) metric dimension of $G$ and it is denoted by
$\mathrm{dim}(G)$ (resp. $\mathrm{edim}(G)$). The \emph{cyclomatic number}
$c(G)$ of a graph $G$ is defined by $c(G)=\left\vert E(G)\right\vert
-\left\vert V(G)\right\vert +1.$ A $\Theta$\emph{-graph }is any graph $G$ with
precisely two vertices of degree $3$ and all other vertices of degree $2$.

The concept of vertex metric dimension was introduced related to the study of
navigation systems \cite{HararyVertex} and the landmarks in networks
\cite{KhullerVertex}. Various aspects of this metric dimension have been
studied since it was first introduced \cite{BuczkowskiVertex, ChartrandVertex,
UnicyclicDudenko2, UnicyclicDudenko4, FehrVertex, KleinVertex, MelterVertex,
UnicyclicPoisson}. As it was noticed recently in \cite{TratnikEdge}, there are
graphs in which none of the smallest vertex metric generators distinguishes
all pairs of edges. This motivated the introduction of a new variant of metric
dimension, namely the edge metric dimension. Even though it is newer than the
vertex metric dimension, the edge metric dimension also atracted interest
\cite{GenesonEdge, HuangApproximationEdge, Klavzar, Knor, PeterinEdge,
ZhangGaoEdge, ZhuEdge, ZubrilinaEdge}.A nice survey of the topic of metric
dimension is given in \cite{YeroSurvey}.

Particularly relevant for this paper is the line of investigation from papers
\cite{SedSkreExtensionCactus, SedSkreLeaflessCacti} where graphs with edge
disjoint cycles, also called \emph{cactus graphs} or \emph{cacti}, were
studied. In \cite{SedSkreExtensionCactus} the upper bounds $\mathrm{dim}%
(G)\leq L(G)+2c(G)$ and $\mathrm{edim}(G)\leq L(G)+2c(G)$ are established to
hold for all cacti, where $L(G)$ is an invariant that depends on the presence
of leaves in $G.$ Moreover, the following conjectures were proposed for
general graphs.

\begin{conjecture}
Let $G$ be a connected graph. Then, $\mathrm{dim}(G)\leq L(G)+2c(G).$
\end{conjecture}

\begin{conjecture}
Let $G$ be a connected graph. Then, $\mathrm{edim}(G)\leq L(G)+2c(G).$
\end{conjecture}

Since the attainment of the bound in the class of cactus graphs depends on the
presence of leaves, leafless cacti and general graphs without leaves were
further investigated in \cite{SedSkreLeaflessCacti}. It was established that
for leafless cacti the upper bound decreases to $2c(G)-1$, and all cacti
attaining this bound were characterized. It was further conjectured that the
same decreased upper bound holds for all leafless graphs, i.e., the following
two conjectures were posed.

\begin{conjecture}
\label{Con_dim_leaves}Let $G\not =C_{n}$ be a graph with minimum degree
$\delta(G)\geq2$. Then, $\mathrm{dim}(G)\leq2c(G)-1.$
\end{conjecture}

\begin{conjecture}
\label{Con_edim_leaves}Let $G\not =C_{n}$ be a graph with minimum degree
$\delta(G)\geq2$. Then, $\mathrm{edim}(G)\leq2c(G)-1.$
\end{conjecture}

To support these conjectures, it was established in
\cite{SedSkreLeaflessCacti} that they hold for all graphs with $\delta
(G)\geq3$ with the strict inequality. Moreover, additional results for graphs
with $\delta(G)=2$ were also established, but let us first define all involved notions.

A set $S\subseteq V(G)$ is called a \emph{vertex cut} if $G-S$ is not
connected or it is trivial. A vertex $v$ is called a \emph{cut vertex}, if
$S=\{v\}$ is a vertex cut. The \emph{(vertex) connectivity} of a graph $G$ is
the size of the smallest vertex cut in $G$ and we denote it by $\kappa(G).$ A
graph $G$ is said to be $k$\emph{-connected} if $\kappa(G)\geq k.$ Any maximal
$2$-connected subgraph of $G$ is called a \emph{block} of $G$. If a block
$G_{i}$ contains at least three vertices, then $G_{i}$ is said to be
\emph{non-trivial}.

In \cite{SedSkreLeaflessCacti} it was established that for $\delta(G)=2$ the
problem can be reduced to $2$-connected graphs, i.e., it was shown that if
Conjecture \ref{Con_dim_leaves} (resp. Conjecture \ref{Con_edim_leaves}) holds
for $2$-connected graphs, then it holds in general. Moreover, considering when
the upper bound is attained, the following claim was established.

\begin{lemma}
\label{Cor_attained} Let $G\not =C_{n}$ be a graph with $\delta(G)\geq2$. If
$\mathrm{dim}(G_{i})<2c(G_{i})-1$ (resp. $\mathrm{edim}(G_{i})<2c(G_{i})-1$)
for a block $G_{i}$ of $G$ distinct from a cycle or there exist two
vertex-disjoint non-trivial blocks $G_{j}$ and $G_{k}$ in $G$, then
$\mathrm{dim}(G)<2c(G)-1$ (resp. $\mathrm{edim}(G)<2c(G)-1$).
\end{lemma}

In this paper, we consider $2$-connected graphs that attain the bound of
Conjectures \ref{Con_dim_leaves} and \ref{Con_edim_leaves}. In particular, we
study $\Theta$-graphs, as they are the simplest $2$-connected graphs distinct
from cycles. We show that the upper bound $2c(G)-1$ holds for both metric
dimensions of $\Theta$-graphs. Since for all $\Theta$-graphs the value of
cyclomatic number equals $2,$ to prove the conjectures it is sufficient to
prove that for all such graphs metric dimensions are bounded above by $3$. We
also characterize all $\Theta$-graphs for which the bounds are attained. The
paper is concluded with the conjectures that the already known extremal
leafless cacti from \cite{SedSkreLeaflessCacti} and the extremal $\Theta
$-graphs established in this paper are the only leafless graphs for which the
bound $2c(G)-1$ is attained. For these conjectures we also established that
they reduce to the same problem on the class of $2$-connected graphs. Similar
results for yet another variant of metric dimension, so called mixed metric
dimension, were already reported in \cite{SedSkreTheta, SedSkrekMixed}.

\section{$\Theta$-graphs with metric dimensions equal to $3$}

So, let us first introduce a necessary notation for $\Theta$-graphs. Let $G$
be a $\Theta$-graph, by $u$ and $v$ we denote the two vertices of degree $3$
in $G.$ Notice that there are three distinct paths in $G$ connecting $u$ and
$v,$ we will denote them by $P_{1}=u_{0}u_{1}\cdots u_{p},$ $P_{2}=v_{0}%
v_{1}\cdots v_{q}$ and $P_{3}=w_{0}w_{1}\cdots w_{r},$ so that $u_{0}%
=v_{0}=w_{0}=u$, $u_{p}=v_{q}=w_{r}=v$ and $p\leq q\leq r$. The cycle in $G$
induced by paths $P_{i}$ and $P_{j}$ will be denoted by $C_{ij}$. A $\Theta
$-graph in which paths $P_{1},$ $P_{2}$ and $P_{3}$ are of lengths $p,$ $q,$
and $r$ respectively, is denoted by $\Theta_{p,q,r}$.

\begin{lemma}
\label{Lemma_vertexTheta} Let $G=\Theta_{p,p,p}$ or $\Theta_{p,p,p+2}$ with
$p\geq2$. Then $\mathrm{dim}(G)\geq3$.
\end{lemma}

\begin{proof}
Let $S\subseteq V(G)$ be a set of vertices in $G$ such that $\left\vert
S\right\vert =2$. It is sufficient to show that $S$ is not a vertex metric
generator. First, if $S=\{u,v\},$ then $u_{1}$ and $v_{1}$ are not
distinguished by $S,$ so we can assume $v\not \in S$. Now, let us consider the
case $S\subseteq V(P_{i})$ for some $i\in\{1,2,3\}.$ Assume first $S\subseteq
V(P_{3}).$ Since $P_{1}$ and $P_{2}$ are of equal length, the distance of
$u_{1}$ and $v_{1}$ to all vertices of $P_{3}$ is the same, hence $S$ does not
distingush $u_{1}$ and $v_{1}.$ Let us now assume $S\subseteq V(P_{1})$ and
let us consider vertices $v_{1}$ and $w_{1}.$ Notice that a shortest path from
both $v_{1}$ and $w_{1}$ to all vertices of $P_{1}$ leads through $u.$ This
implies that the distance from $v_{1}$ and $w_{1}$ to all the vertices of
$P_{1}$ is the same, so a set $S\subseteq V(P_{1})$ would not distinguish
$v_{1}$ and $w_{1}.$ The same reasoning goes for $S\subseteq V(P_{2}),$ so we
may assume that $S\not \subseteq V(P_{i})$ for every $i=1,2,3.$

Now, denote by $s_{1}$ and $s_{2}$ the two elements of $S$. Then $s_{1}$ and
$s_{2}$ are internal vertices of paths $P_{i}$ and $P_{j}$, respectively,
where $i\not =j$. We distinguish two cases.\begin{figure}[h]
\begin{center}
$%
\begin{array}
[c]{ll}%
\text{a) \raisebox{-1\height}{\includegraphics[scale=0.6]{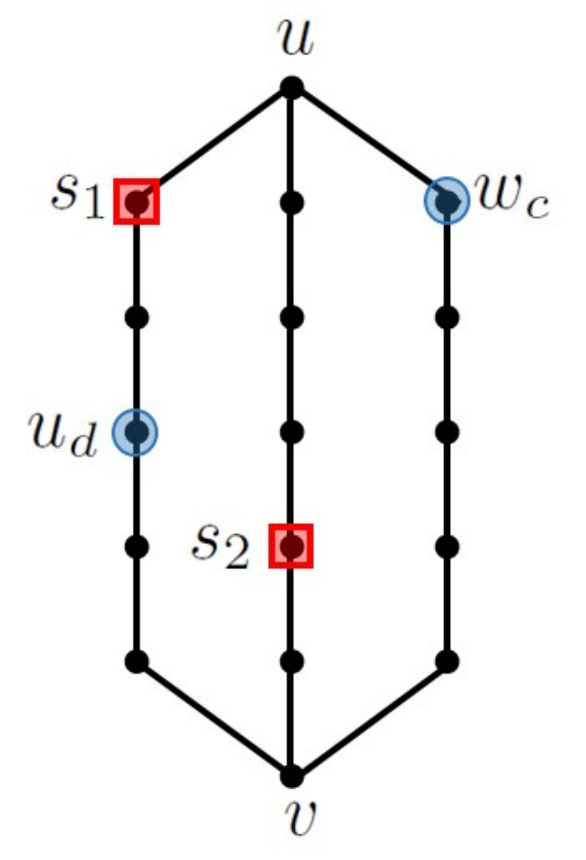}}} &
\text{b) \raisebox{-1\height}{\includegraphics[scale=0.6]{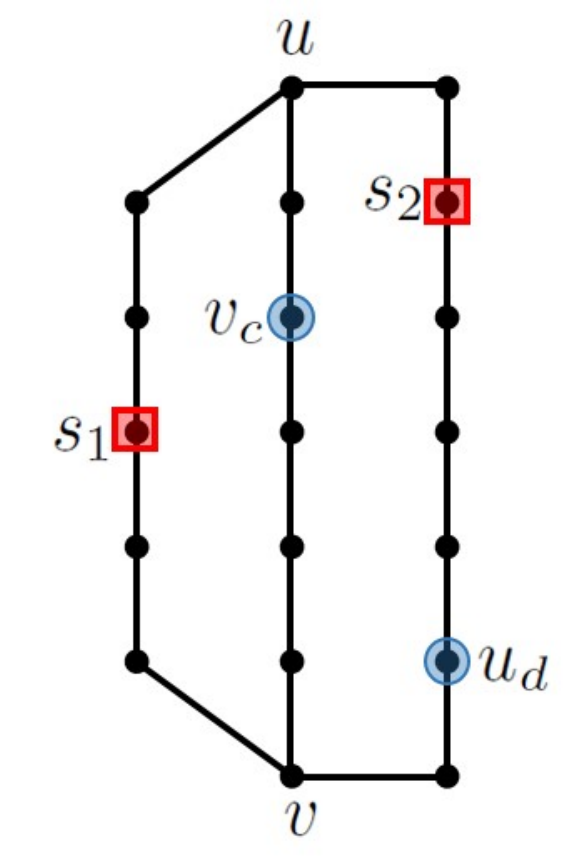}}}%
\end{array}
$
\end{center}
\caption{A set $S=\{s_{1},s_{2}\}$ in the proof of Lemma
\ref{Lemma_vertexTheta}: a) case when $s_{1}\in V(P_{1})$ and $s_{2}\in
V(P_{2})$ with $p=6,$ $d_{1}=1,$ $d_{2}=4,$ $a=5,$ $b=7,$ $c=1$ and $d=3$ in
which $u_{d}$ and $w_{c}$ are not distinguished by $S;$ b) case when $s_{1}\in
V(P_{1})$ and $s_{2}\in V(P_{3})$ with $p=6,$ $d_{1}=3,$ $d_{2}=2,$ $a=5,$
$b=9,$ $c=2$ and $d=8,$ where $u_{d}$ and $v_{c}$ are not distinguished by
$S.$}%
\label{Fig_cases}%
\end{figure}

\medskip\noindent\textbf{Case 1:} $s_{1}\in V(P_{1})$\emph{ and }$s_{2}\in
V(P_{2})$\emph{.} Let us denote $d_{1}=d(s_{1},u),$ $d_{2}=d(s_{2},u),$
$a=d_{1}+d_{2}$ and $b=2p-a$. If $a=b,$ then $s_{1}$ and $s_{2}$ form an
antipodal pair on $C_{12}$, which implies that two neighbours of $s_{1}$ are
not distinguished by $S$. So, without loss of generality we may assume $a<p$
and $d_{1}\leq d_{2}$. Since $a+b=2p,$ it follows that $a$ and $b$ are of the
same parity, hence $b-a$ is a positive even number. Therefore, we can define
$c=(b-a)/2$ and we know that $c$ is a positive integer. Let $d=2d_{1}+c$.
Notice that
\[
c<d=2d_{1}+c\leq a+c=\frac{a}{2}+\frac{b}{2}=p.
\]
So there exist interior vertices $u_{d}\in P_{1}$ and $w_{c}\in P_{3}$, see
Figure \ref{Fig_cases} a).

Now we prove that $u_{d}$ and $w_{c}$ are not distinguished by $S$. Notice
that $d(u_{d},s_{1})=d-d_{1}=d_{1}+c.$ Since
\[
c+d_{1}=\frac{b}{2}-\frac{a}{2}+d_{1}\leq\frac{b}{2}=p-\frac{a}{2}<p,
\]
we have $d(w_{c},s_{1})=c+d_{1}$, and so $u_{d}$ and $w_{c}$ are not
distinguished by $s_{1}$. As for $s_{2}$, notice that
\[
c+d_{2}<\frac{b-a}{2}+a=p,
\]
so we have $d(w_{c},s_{2})=c+d_{2}.$ Also, we have
\[
d_{2}+d=d_{2}+2d_{1}+c=a+d_{1}+\frac{b-a}{2}=p+d_{1}>p,
\]
which implies
\[
d(u_{d},s_{2})=2p-d-d_{2}=2p-p-d_{1}=p-d_{1}=p-a+d_{2}=c+d_{2}.
\]
We conclude that $u_{d}$ and $w_{c}$ are not distinguished by $s_{2}$ either,
so $S$ is not a vertex metric generator.

\medskip\noindent\textbf{Case 2:} $s_{1}\in V(P_{1})$\emph{ and }$s_{2}\in
V(P_{3})$\emph{.} For $G=\Theta_{p,p,p}$ this case is analogous to the
previous one, so let us assume $G=\Theta_{p,p,p+2}$. Again, denote
$d_{1}=d(u,s_{1})$, $d_{2}=d(u,s_{2}),$ $a=d_{1}+d_{2}$ and $b=2p+2-a$. If
$a=b,$ then $s_{1}$ and $s_{2}$ are antipodal on $C_{13}$, so the two
neighbors of $s_{1}$ are not distinguished by $S$. Hence, without loss of
generality we may assume $a<b.$ Let us denote $c=(b-a)/2.$ Since $a+b=2p+2$ we
know that $a$ and $b$ are of the same parity, so $b-a$ is a positive integer.
Consequently, also $c$ is a positive integer.

First, since $s_{1}$ and $s_{2}$ are internal vertices of paths $P_{1}$ and
$P_{3}$ respectively, we have $a=d_{1}+d_{2}\geq2$. This yields%
\[
c=\frac{b-a}{2}=\frac{a+b}{2}-a=p+1-a\leq p-1.
\]
Hence, there exists an interior vertex $v_{c}\in V(P_{2}),$ as it is shown in
Figure \ref{Fig_cases} b). Also, notice that%
\[
d_{1}+c<a+\frac{b-a}{2}=\frac{a+b}{2}=p+1,
\]
which implies $d(v_{c},s_{1})=d_{1}+c.$

Now, let $d=2d_{1}+c.$ If $d\leq p$ we consider the vertex $u_{d}\in
V(P_{1}),$ otherwise for the sake of simplicity we denote $u_{d}=w_{2p+2-d}$,
see Figure \ref{Fig_cases} b). We have already shown $d_{1}+c<p+1$, which
yields
\[
d-d_{1}=d_{1}+c<p+1,
\]
and so $d(u_{d},s_{1})=d-d_{1}=d_{1}+c=d(v_{c},s_{1})$. Hence, $u_{d}$ and
$v_{c}$ are not distinguished by $s_{1}.$ It remains to prove that $u_{d}$ and
$v_{c}$ are not distinguished by $s_{2}$ either. For that purpose, notice that%
\[
c+d_{2}<c+a=\frac{b-a}{2}+a=\frac{a+b}{2}=p+1,
\]
which implies $d(v_{c},s_{2})=c+d_{2}.$ Also, notice that
\begin{align*}
2p+2-d-d_{2}  &  =a+b-2d_{1}-c-d_{2}=a+b-d_{1}-\frac{b-a}{2}-(d_{1}+d_{2})\\
&  =\frac{a+b}{2}-d_{1}=p+1-d_{1}<p+1,
\end{align*}
which implies
\begin{align*}
d(s_{2},u_{d})  &  =2p+2-d-d_{2}=\frac{a+b}{2}-d_{1}=\frac{a+b}{2}%
-a+a-d_{1}=\\
&  =\frac{b-a}{2}+d_{2}=c+d_{2}=d(v_{c},s_{2}).
\end{align*}
Therefore, vertices $v_{c}$ and $u_{d}$ are not distinguished by $s_{2}$
either, hence we conclude that $S$ is not a vertex metric generator.
\end{proof}

Now, a subgraph $H$ of a graph $G$ is an \emph{isometric} subgraph if
$d_{H}(u,v)=d_{G}(u,v)$ for every pair of vertices $u,v\in V(H).$
Consequently, if a pair of vertices is distinguished by $S\cap V(H)$ in $H,$
then it is distinguished by $S$ in $G$ too.

\begin{lemma}
\label{Lemma_vertexGenerators} Let $G=\Theta_{p,p,p}$ or $\Theta_{p,p,p+2}$
with $p\geq\not 2  $. Then for any $a\in V(G),$ there are $b,c\in V(G)$ such
that $S=\{a,b,c\}$ is a vertex metric generator in $G$.
\end{lemma}

\begin{proof}
First, notice that every cycle $C_{ij}$ of $G$ is an isometric subgraph in
$G.$ We say that a set $S\subseteq V(G)$ is \emph{nice}, if for every cycle
$C_{ij}$ of $G$ it holds that $S\cap V(C_{ij})$ contains two vertices which do
not form an antipodal pair in $C_{ij}$. We first show that any nice set $S$ is
a vertex metric generator in $G.$ In order to see this, let $x$ and
$x^{\prime}$ be a pair of vertices from $G.$ Notice that $x$ and $x^{\prime}$
belong to at least one cycle $C_{ij}$ in $G$. Since $S$ is nice, $S\cap
V(C_{ij})$ contains two vertices which are not antipodal in $C_{ij}$, which
implies that $S\cap V(C_{ij})$ is a vertex metric generator in $C_{ij}.$
Therefore, $x$ and $x^{\prime}$ are distinguished by $S\cap V(C_{ij})$ in
$C_{ij}.$ Since $C_{ij}$ is an isometric subgraph of $G,$ this further implies
that $x$ and $x^{\prime}$ are distinguished by $S$ in $G,$ so $S$ is a vertex
metric generator of $G$. To complete the proof, for every $a\in V(G)$ we
extend $a$ to a nice set.

Let us assume $G=\Theta_{p,p,p}.$ If $p\leq3,$ the set $S=\{u,v_{1},w_{1}\}$
is a nice set in $G$. Therefore, $S$ is a vertex metric generator, which due
to symmetry of $G$ proves the claim. So, let us assume that $p\geq4$. By
symmetry, we may assume that $a=u_{i}$, where $0\leq i\leq\left\lfloor
p/2\right\rfloor $. But then $S_{i}=\{u_{i},v_{1},w_{1}\}$ is a nice set in
$G$.

Assume now that $G=\Theta_{p,p,p+2}.$ If $p=2,$ it is easy to see that sets
$S=\{u,v_{1},w_{1}\}$ and $S=\{u,v_{1},w_{2}\}$ are nice in $G,$ which due to
symmetry of $G$ proves the claim. If $p>2,$ then due to symmetry of $G$ it is
sufficient to prove the claim for $a=u_{i}$ where $0\leq i\leq\left\lfloor
p/2\right\rfloor $ and for $a=w_{j}$ where $1\leq j\leq\left\lfloor
p/2\right\rfloor +1.$ If $a=u_{i}$ for $i\leq\left\lfloor p/2\right\rfloor ,$
then $S=\{u_{i},v_{1},w_{1}\}$ is nice in $G$. On the other hand, if $a=w_{j}$
for $j\leq\left\lfloor p/2\right\rfloor +1$, then $S=\{u_{1},v_{1},w_{j}\}$ is
nice in $G$.
\end{proof}

By Lemmas~{\ref{Lemma_vertexTheta}} and~{\ref{Lemma_vertexGenerators}} the
following statement holds.

\begin{theorem}
For $p\geq2,$ it holds that $\mathrm{dim}(\Theta_{p,p,p})=\mathrm{dim}%
(\Theta_{p,p,p+2})=3$.
\end{theorem}

Since in any $\Theta$-graph $G$ it holds that $L(G)=0$ and $c(G)=2,$ the above
theorem gives the following corollary.

\begin{corollary}
We have $\mathrm{dim}(\Theta_{p,p,p})=\mathrm{dim}(\Theta_{p,p,p+2})=2c(G)-1$.
\end{corollary}

Hence, for $\Theta_{p,p,p}$ and $\Theta_{p,p,p+2}$ the bound from
Conjecture~{\ref{Con_dim_leaves}} holds with equality. Similarly, when
considering the edge metric dimension of $\Theta$-graphs, we have the following.

\begin{lemma}
\label{Lemma_edgeGenerators}Let $G=\Theta_{1,2,2}$ or $\Theta_{p,p,q}$ with
$2\leq p\leq3$ and $p\leq q\leq p+2.$ Then for any $a\in V(G),$ there are
$b,c\in V(G)$ such that $S=\{a,b,c\}$ is an edge metric generator in $G$.
\end{lemma}

\begin{proof}
As $p=2$ or $3$ and $q\in\{p,p+1,p+2\},$ the problem is finite. To avoid a
tedious proof, the statement was easily verified by a computer by checking all
sets $S\subseteq V(G)$ of cardinality $3$.
\end{proof}

\begin{proposition}
Let $G=\Theta_{1,2,2}$ or $\Theta_{p,p,q}$ with $2\leq p\leq3$ and $p\leq
q\leq p+2$. Then $\mathrm{edim}(G)=3$.
\end{proposition}

\begin{proof}
Similarly as before, by a computer we checked easily that there is no edge
metric generator of size two. Then the claim follows from
Lemma~{\ref{Lemma_edgeGenerators}}.
\end{proof}

\begin{corollary}
Let $G=\Theta_{1,2,2}$ or $G=\Theta_{p,p,q}$ for $2\leq p\leq3$ and $p\leq
q\leq p+2$. Then $\mathrm{edim}(G)=2c(G)-1.$
\end{corollary}

\section{$\Theta$-graphs with metric dimensions equal to $2$}

In this section we show that all remaining $\Theta$-graphs, i.e., all $\Theta
$-graphs not mentioned in the previous section, have the vertex (resp. the
edge) metric dimension equal to $2$.
We first consider the vertex metric dimension. For all remaining $\Theta
$-graphs we show that there is a set $S$ of cardinality two which is a vertex
metric generator, see Figure \ref{Fig_VertexGenerators2}.

\begin{lemma}
\label{Lemma_dim2} Let $G=\Theta_{p,q,r}$, where $p\leq q\leq r,$ and let $S$
be a set of vertices in $G,$ defined in the following way:

\begin{enumerate}
\item[i)] if one of $p$, $q$, $r$ is odd and at least $3$ and one of $p$, $q$,
$r$ is even, say $q\geq3$ is odd and $r$ is even, then $S=\{v_{(q-1)/2}%
,w_{r/2}\}$;

\item[ii)] if $p=1$ and both $q$ and $r$ are even, then $S=\{u,w_{r/2}\}$;

\item[iii)] if all $p$, $q$, $r$ are even and $q\notin\{p,p+2\},$ then
$S=\{v_{1},w_{r/2}\}$;

\item[iv)] if all of $p$, $q$, $r$ are even, $q\in\{p,p+2\}$ and $r\geq p+4,$
then $S=\{v_{q/2},w_{1}\}$;

\item[v)] if all $p$, $q$, $r$ are even and $q=r=p+2,$ then $S=\{v_{1}%
,w_{1}\}$;

\item[vi)] if all $p$, $q$, $r$ are odd and $q\notin\{p,p+2\}$, then
$S=\{v_{1},w_{(r-1)/2}\}$;

\item[vii)] if all $p$, $q$, $r$ are odd, $q\in\{p,p+2\}$ and $r\geq p+4$,
then $S=\{v_{(q-1)/2},w_{1}\}$;

\item[viii)] if all $p$, $q$, $r$ are odd and $q=r=p+2,$ then $S=\{v_{1}%
,w_{1}\}$.
\end{enumerate}

\noindent Then $S$ is a vertex metric generator in $G.$
\end{lemma}

\begin{proof}
First we introduce some notation. For a vertex $a\in V(G)$ we denote by
$\mathcal{P}_{a}$ the partition of $V(G)$ according to the distances from $a$.
That is, if $x,x^{\prime}$ are in the same set of $\mathcal{P}_{a}$, then
$d(a,x)=d(a,x^{\prime})$. To prove that $S=\{a,b\}$ is a vertex metric
generator in $G$, it suffices to show that $d(b,x)\neq d(b,x^{\prime})$ for
every pair of vertices $x,x^{\prime}$ from a common set of $\mathcal{P}_{a}$.
Proceeding by way of contradiction, if $d(b,x)=d(b,x^{\prime})$ then the
shortest path from $b$ to $x$ cannot contain a path from $b$ to $x^{\prime}$
and vice versa. This simplifies our consideration since $\Theta_{p,q,r}$
contains only two branching vertices (i.e., vertices of degree at least $3$).
Let us now consider each of the eight cases separately.

\medskip i) For the vertex $w_{r/2}\in S,$ we have
\[
\mathcal{P}_{w_{\frac{r}{2}}}=(\{w_{\frac{r}{2}}\},\{w_{i},w_{r-i}%
\}_{i=0}^{\frac{r}{2}-1},\{u_{i},v_{i},u_{p-i},v_{q-i}\}_{i=1}^{\lfloor
\frac{p}{2}\rfloor},\{v_{i},v_{q-i}\}_{i=\lfloor\frac{p}{2}\rfloor+1}%
^{\lfloor\frac{q}{2}\rfloor}).
\]
We have to show that the other vertex of $S,$ i.e. $v_{(q-1)/2},$
distinguishes all pairs of vertices from a common set of $\mathcal{P}%
_{w_{\frac{r}{2}}}.$ The first type of set in $\mathcal{P}_{w_{\frac{r}{2}}}$
which contains at least one pair of vertices is $\{w_{i},w_{r-i}\},$ so we
have to show that $w_{i}$ and $w_{r-i}$ are distinguished by $v_{(q-1)/2},$
and that follows from
\[
d(v_{\frac{q-1}{2}},w_{i})=i+\frac{q-1}{2}<i+\frac{q+1}{2}=d(w_{\frac{q-1}{2}%
},w_{r-i}).
\]
The next set from $\mathcal{P}_{w_{\frac{r}{2}}}$ to consider is of the type
$\{u_{i},v_{i},u_{p-i},v_{q-i}\},$ where we have%
\[
d(v_{\frac{q-1}{2}},v_{i})<d(v_{\frac{q-1}{2}},v_{q-i})<d(v_{\frac{q-1}{2}%
},u_{i})<d(y_{\frac{q-1}{2}},u_{p-i}),
\]
where the last two expressions have place only if $i\leq\lfloor\frac{p}%
{2}\rfloor$. Therefore, all pairs of vertices from that set are distinguished
by $v_{(q-1)/2}\in S$. Notice that the inequality covers also the last type of
set from $\mathcal{P}_{w_{\frac{r}{2}}}$. Also observe that we did not use the
fact that $p\leq q\leq r$ here, so the proof covers all cases when one of
$p,q,r$ is odd and at least $3$ and one of $p,q,r$ is even.

\medskip ii) Analogously as in i) we have
\[
\mathcal{P}_{w_{\frac{r}{2}}}=(\{w_{\frac{r}{2}}\},\{w_{i},w_{r-i}%
\}_{i=0}^{\frac{r}{2}-1},\{v_{i},v_{q-i}\}_{i=1}^{\frac{q}{2}-1},\{v_{\frac
{q}{2}}\}).
\]
It remains to show that $u$ distinguishes all pairs of vertices which belong
to a common set of $\mathcal{P}_{w_{\frac{r}{2}}}$. This is seen from
$d(u,w_{i})=i<i+1=d(u,w_{r-i})$ and $d(u,v_{i})=i<i+1=d(u,v_{q-i})$.

\medskip iii) We have
\[
\mathcal{P}_{w_{\frac{r}{2}}}=(\{w_{\frac{r}{2}}\},\{w_{i},w_{r-i}%
\}_{i=0}^{\frac{r}{2}-1},\{u_{i},v_{i},u_{p-i},v_{q-i}\}_{i=1}^{\frac{p}{2}%
},\{v_{i},v_{q-i}\}_{i=\frac{p}{2}+1}^{\frac{q}{2}}).
\]
(Observe that, the third set has just three vertices if $i=p/2$, and the last
set has just one vertex if $i=q/2$.) We show that $v_{1}\in S$ distinguishes
all pairs of vertices from a common set of $\mathcal{P}_{w_{\frac{r}{2}}}$.
Regarding set $\{w_{i},w_{r-i}\}$, notice that $d(v_{1},w_{i}%
)=i+1<1+p+i=d(v_{1},w_{r-i}).$ The next sets of $\mathcal{P}_{w_{\frac{r}{2}}%
}$ are of the form $\{u_{i},v_{i},u_{p-i},v_{q-i}\}$ where
\[
d(v_{1},v_{i})=i-1<i+1=d(v_{1},u_{i}),
\]
and assuming that $v_{1}$ does not distinguish the other possible pairs leads
to a contradiction, namely $d(v_{1},u_{i})=d(v_{1},u_{p-i})$ implies
$i+1=q-1+i$ and $q=2$, a contradiction; $d(v_{1},u_{i})=d(v_{1},v_{q-i})$
implies $i+1=q-i-1$ and $i=q/2-1$, but such $u_{i}$ exists only if $q\leq
p+2$, a contradiction; $d(v_{1},v_{i})=d(v_{1},u_{p-i})$ implies $i-1=p-i+1$
and $i=p/2+1$, but such $i$ is over the limit for this set; $d(v_{1}%
,v_{i})=d(v_{1},v_{q-i})$ implies $i-1=1+p+i$ or simplified $p=-2$, a
contradiction; $d(v_{1},u_{p-i})=d(v_{1},v_{q-i})$ implies $1+p-i=q-i-1$ and
$q=p+2$, a contradiction.

For the last set of $\mathcal{P}_{w_{\frac{r}{2}}}$ we have $d(v_{1}%
,v_{i})=i-1<q-i-1=d(v_{1},v_{i})$ whenever $i<q/2$, and for $i=q/2$ the set is
a singleton.

\medskip iv) For $v_{q/2}\in S$ we have
\[
\mathcal{P}_{v_{\frac{q}{2}}}=(\{v_{\frac{q}{2}}\},\{v_{i},v_{q-i}%
\}_{i=0}^{\frac{q}{2}-1},\{u_{i},w_{i},u_{p-i},w_{r-i}\}_{i=1}^{\frac{p}{2}%
},\{w_{i},w_{r-i}\}_{i=\frac{p}{2}+1}^{\frac{r}{2}}).
\]
Now we consider the distances from $w_{1}\in S$. Assuming $d(w_{1}%
,v_{i})=d(w_{1},v_{q-i})$ implies $i+1=p+1+i$ so $p=0$, a contradiction.

The next set to consider is of the form $\{u_{i},w_{i},u_{p-i},w_{r-i}\}$. We
have
\[
d(w_{1},w_{i})=i-1<\min\{d(w_{1},u_{i}),d(w_{1},u_{p-i}),d(w_{1},w_{r-i})\},
\]
which resolves three of the six possible pairs of vertices. For all other
possible pairs we will assume that they are not distinguished by $w_{1}$ and
show that it leads to contradiction. Namely, $d(w_{1},u_{i})=d(w_{1},u_{p-i})$
implies $i+1=1+q+i$ or simplified $q=0$, a contradiction; $d(w_{1}%
,u_{i})=d(w_{1},w_{r-i})$ implies $i+1=r-i-1$ which reduces to $i=r/2-1$, but
such $i$ exceeds the limit for this set since $r\geq p+4$; $d(w_{1}%
,u_{p-i})=d(w_{1},w_{r-i})$ implies $1+p-i=r-i-1$ which reduces to $r=p+2$, a contradiction.

For the last set of $\mathcal{P}_{v_{\frac{q}{2}}}$ if $w_{i}\ne w_{r-i}$ and
$d(w_{1},w_{i})=d(w_{1},w_{r-i})$, then $i-1=p+1+i$ and $p=-2$, a contradiction.

\medskip v) Partition for $v_{1}\in S$ is
\[
\mathcal{P}_{v_{1}}=(\{v_{1}\},\{u,v_{2}\},\{u_{i},v_{i+2},w_{i}\}_{i=1}%
^{p-1},\{v,w_{p}\},\{w_{p+1}\}).
\]
and for distances from $w_{1}\in S$ we have
\begin{align*}
d(w_{1},u)  &  =1<3=d(w_{1},v_{2}),\\
d(w_{1},w_{i})  &  =i-1<d(w_{1},u_{i})=i+1<d(w_{1},v_{i+2})=i+3,\\
d(w_{1},w_{p})  &  =p-1<p+1=d(w_{1},v).
\end{align*}

\medskip vi) For $w_{(r-1)/2}\in S$ we have
\[
\mathcal{P}_{w_{\frac{r-1}{2}}}=(\{w_{\frac{r-1}{2}}\},\{w_{i},w_{r-i-1}%
\}_{i=0}^{\frac{r-3}{2}},\{u_{1},v_{1},v\},\{u_{i},v_{i},u_{p-i+1}%
,v_{q-i+1}\}_{i=2}^{\frac{p+1}{2}},\{v_{i},v_{q-i+1}\}_{i=\frac{p+3}{2}%
}^{\frac{q+1}{2}}).
\]
Now consider the distances from $v_{1}\in S$. Assuming $d(v_{1},w_{i}%
)=d(v_{1},w_{r-i+1})$ implies $i+1=p+1+i+1$ which reduces to $p=-1$, a contradiction.

In the next set $\{u_{1},v_{1},v\}$ of $\mathcal{P}_{w_{\frac{r-1}{2}}}$ there
are three possible pairs of vertices, for which we have%
\[
d(v_{1},v_{1})=0<d(v_{1},u_{1})=2<d(v_{1},v)=p+1,
\]
where the last inequality holds if $p>1$, otherwise $u_{1}=v$ so there is no
pair to be distinguished.

The next set from $\mathcal{P}_{w_{\frac{r-1}{2}}}$ is of the type
$\{u_{i},v_{i},u_{p-i+1},v_{q-i+1}\}$, where we first have $d(v_{1}%
,v_{i})=i-1<i+1=d(v_{1},u_{i})$, so the pair $u_{i},v_{i}$ is distinguished by
$v_{1}\in S$. For all remaining pairs of vertices from that set, we will show
that assuming they are not distinguished by $s_{1}\in S$ leads to a
contradiction. If $d(v_{1},u_{i})=d(v_{1},u_{p-i+1})$ then $i+1=q-1+i-1$ and
$q=3$, a contradiction; if $d(v_{1},u_{i})=d(v_{1},v_{q-i+1})$ then
$i+1=q-i+1-1$ which reduces to $i=(q-1)/2$, but such $i$ exceeds the limit
since $q>p+2$; if $d(v_{1},v_{i})=d(v_{1},u_{p-i+1})$ then $i-1=1+p-i+1$ and
therefore $i=(p+3)/2$, but such $i$ exceeds the limit; if $d(v_{1}%
,v_{i})=d(v_{1},v_{q-i+1})$ then $i-1=1+p+i-1$ which reduces to $p=-1$, a
contradiction; finally if $d(v_{1},u_{p-i+1})=d(v_{1},v_{q-i+1})$ then
$1+p-i+1=q-i+1-1$ and therefore $q=p+2$, a contradiction.

As for the last type of set in $\mathcal{P}_{w_{\frac{r-1}{2}}}$, if $v_{i}\ne
v_{q-i+1}$ and $d(v_{1},v_{i})=d(v_{1},v_{q-i+1})$ then $i-1=1+p+i-1$ and so
$p=-1$, a contradiction.

\medskip vii) Observe that
\[
\mathcal{P}_{v_{\frac{q-1}{2}}}=(\{v_{\frac{q-1}{2}}\},\{v_{i},v_{q-i-1}%
\}_{i=0}^{\frac{q-3}{2}},\{u_{1},w_{1},v\},\{u_{i},w_{i},u_{p-i+1}%
,w_{r-i+1}\}_{i=2}^{\frac{p+1}{2}},\{w_{i},w_{r-i+1}\}_{i=\frac{p+3}{2}%
}^{\frac{r+1}{2}}).
\]
Now we consider the distances from $w_{1}\in S$. If $d(w_{1},v_{i}%
)=d(w_{1},v_{q-i-1})$ then $i+1=p+1+i+1$ and $p=-1$, a contradiction.

As for the set $\{u_{1},w_{1},v\}\in\mathcal{P}_{v_{\frac{q-1}{2}}},$ we have%
\[
d(w_{1},w_{1})=0<d(w_{1},u_{1})=2<d(w_{1},v)=r-1.
\]
So all three pairs of vertices from this set are distinguished by $w_{1}$.

For the next set of $\mathcal{P}_{v_{\frac{q-1}{2}}}$ we first have
\[
d(w_{1},w_{i})=i-1<\min\{d(w_{1},u_{i}),d(w_{1},u_{p-i+1}),d(w_{1}%
,w_{r-i+1})\},
\]
so $w_{1}$ distinguishes $w_{i}$ from all the other vertices in that set. If
$d(w_{1},u_{i})=d(w_{1},u_{p-i+1})$ then $i+1=1+q+i-1$ and $q=1$, a
contradiction. If $d(w_{1},u_{i})=d(w_{1},w_{r-i+1})$ then $i+1=r-i+1-1$ and
$i=(r-1)/2\geq(p+3)/2$, but such $i$ exceeds the limit. Finally,
$d(w_{1},u_{p-i+1})=d(w_{1},w_{r-i+1})$ implies $1+p-i+1=r-i+1-1$ which
reduces to $r=p+2$, a contradiction.

For the last set of $\mathcal{P}_{v_{\frac{q-1}{2}}}$, if $w_{i}\ne w_{r-i+1}$
and $d(w_{1},w_{i})=d(w_{1},w_{r-i+1})$ then $i-1=1+p+i-1$ and $p=-1$, a contradiction.

\medskip viii) Observe that
\[
\mathcal{P}_{v_{1}}=(\{v_{1}\},\{u,v_{2}\},\{u_{i},v_{i+2},w_{i}\}_{i=1}%
^{p-1},\{v,w_{p}\},\{w_{p+1}\}).
\]
Hence $\mathcal{P}_{v_{1}}$ (and also $\mathcal{P}_{w_{1}}$) does not depend
on the parity of $p$. So analogously as in case v) one can show that
$S=\{v_{1},w_{1}\}$ is a vertex metric generator in this case.
\end{proof}

\begin{figure}[h]
\begin{center}
$%
\begin{array}
[c]{llll}%
\text{i) \raisebox{-1\height}{\includegraphics[scale=0.6]{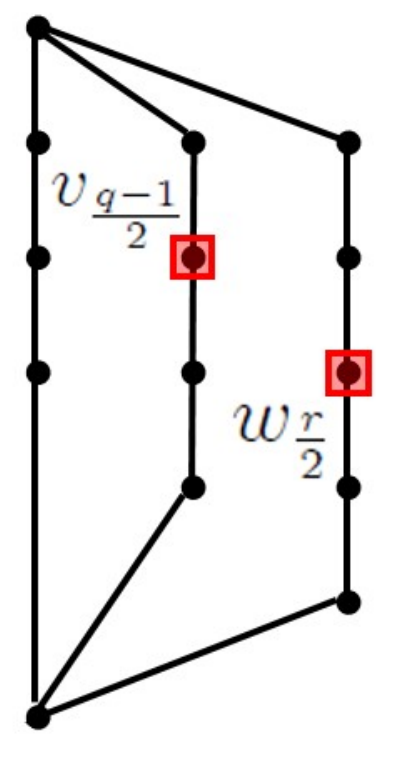}}} &
\text{ii) \raisebox{-1\height}{\includegraphics[scale=0.6]{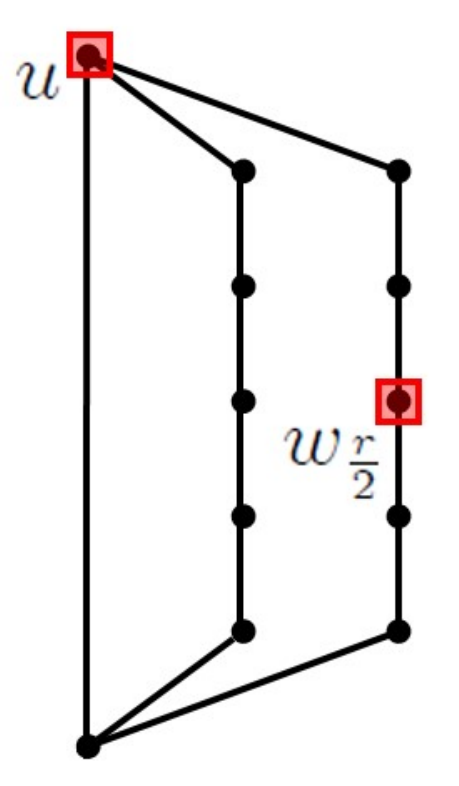}}} &
\text{iii) \raisebox{-1\height}{\includegraphics[scale=0.6]{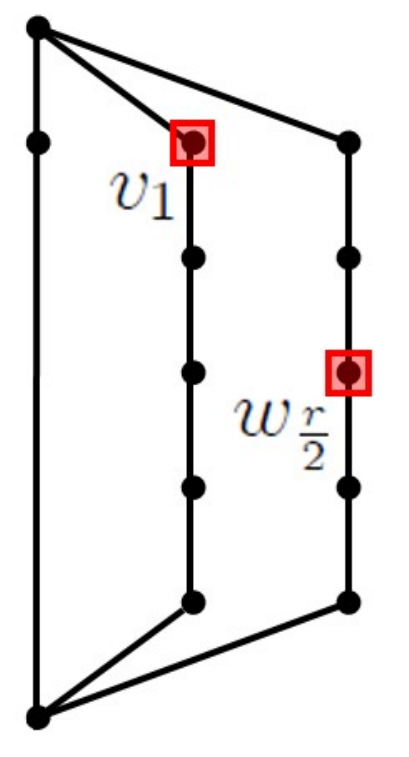}}} &
\text{iv) \raisebox{-1\height}{\includegraphics[scale=0.6]{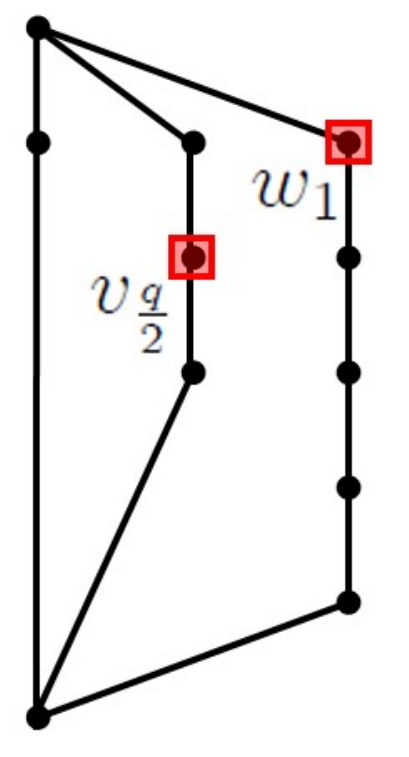}}}\\
\text{v) \raisebox{-1\height}{\includegraphics[scale=0.6]{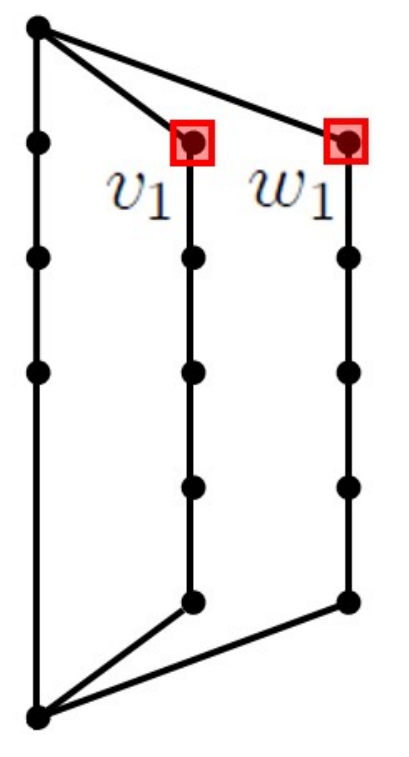}}} &
\text{vi) \raisebox{-1\height}{\includegraphics[scale=0.6]{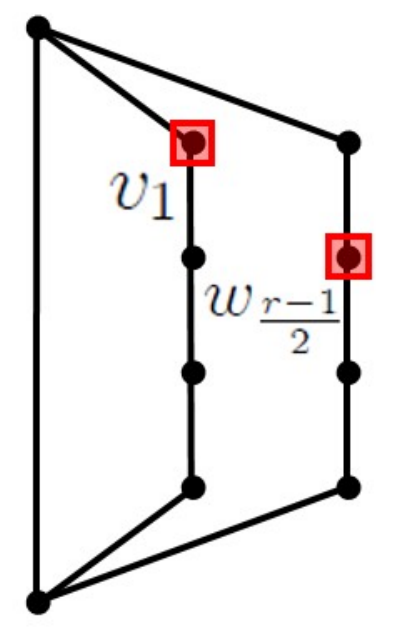}}} &
\text{vii) \raisebox{-1\height}{\includegraphics[scale=0.6]{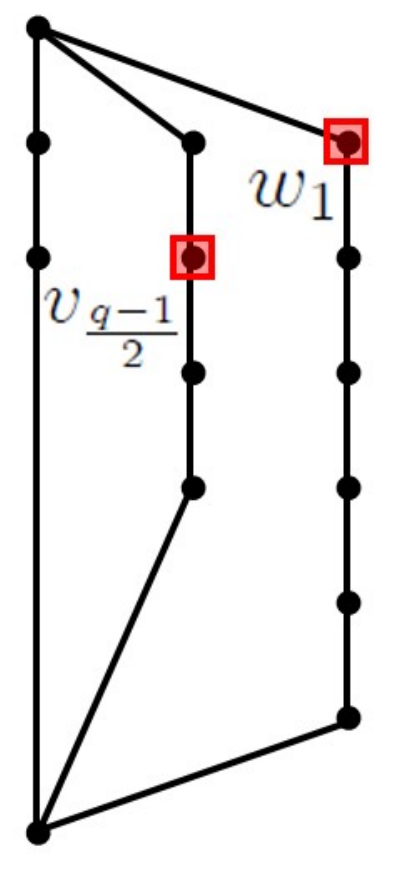}}} &
\text{viii) \raisebox{-1\height}{\includegraphics[scale=0.6]{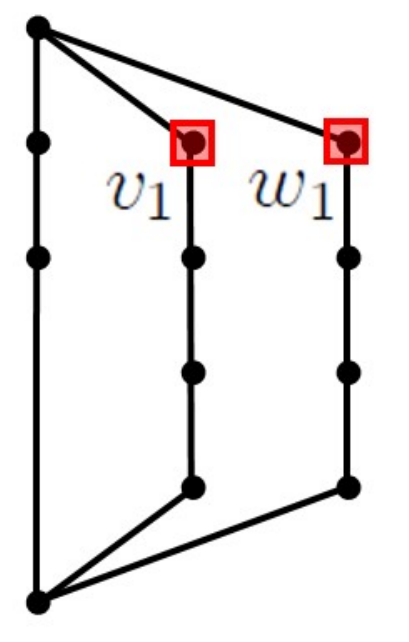}}}%
\end{array}
$
\end{center}
\caption{Vertex metric generators from Lemma \ref{Lemma_dim2}.}%
\label{Fig_VertexGenerators2}%
\end{figure}

Using Lemma~{\ref{Lemma_dim2}} we can prove that all $\Theta$-graphs not
mentioned in the previous section have metric dimension $2$.

\begin{theorem}
\label{thm:dim=2} Let $G$ be a $\Theta$-graph such that $G\not =\Theta
_{p,p,p}$ and $\Theta_{p,p,p+2}$ with $p\geq2.$ Then $\mathrm{dim}(G)=2.$
\end{theorem}

\begin{proof}
It is sufficient to show that Lemma \ref{Lemma_dim2} includes all $\Theta
$-graphs distinct from $\Theta_{p,p,p}$ and $\Theta_{p,p,p+2}$. Cases iii)-v)
of this lemma obviously include all $\Theta$-graphs distinct from
$\Theta_{p,p,p}$ and $\Theta_{p,p,p+2}$ in which all three parameters $p$, $q$
and $r$ are even. Similarly, cases vi)-viii) of the same lemma include all
$\Theta$-graphs in which all three parameters are odd. It remains to show that
cases i)-ii) cover all $\Theta$-graphs in which $p$, $q$ and $r$ do not have a
same parity. In that case at least one of the parameters is odd. If none of
the parameters is equal to one, then Lemma \ref{Lemma_dim2}.i) covers the
cases. If there is parameter equal to $1$, then $p=1$ since $p\leq q\leq r$.
Since $G$ has no parallel edges, $q\geq2$. Hence if one of $q$ and $r$ is odd
then this parameter is at least $3$ and the other parameter is even, which is
covered by Lemma \ref{Lemma_dim2}.i) again. The only remaining case when $p=1$
and both $q$ and $r$ are even is covered by Lemma \ref{Lemma_dim2}.ii).
\end{proof}

As regards the motivating question for this investigation,
Theorem~{\ref{thm:dim=2}} yields the following corollary.

\begin{corollary}
Let $G$ be a $\Theta$-graph such that $G\not =\Theta_{p,p,p}$ and
$G\not =\Theta_{p,p,p+2}$. Then $\mathrm{dim}(G)<2c(G)-1$.
\end{corollary}

Now we consider the edge metric dimension of $\Theta$-graphs. We proceed
analogously as in the case of vertex metric dimension. The edge metric
generators from the following lemma are illustrated in
Figure~{\ref{Fig_EdgeGenerators2}}.

\begin{lemma}
\label{Lemma_edim2} Let $G=\Theta_{p,q,r},$ where $p\leq q\leq r$, and let $S$
be a set of vertices in $G$ defined in a following way:

\begin{enumerate}
\item[i)] if $p<q$, $r\geq3$ and $p+r$ is even, then $S=\{w_{(r-p)/2}%
,w_{(r+p)/2}\}$;

\item[ii)] if $p<q$, $r\geq p+3$ and $p+r$ is odd, then $S=\{w_{\lfloor
(r-p)/2\rfloor},w_{\lceil(r+p)/2\rceil}\}$;

\item[iii)] if $p<q$, $r=p+1$ and $(p,q,r)\neq(1,2,2)$, then $S=\{v_{1}%
,w_{1}\}$;

\item[iv)] if $p=q$ and $p\geq4$, then $S=\{u_{2},v_{1}\}$;

\item[v)] if $p=q$ and $r\geq p+3$, then $S=\{v_{1},w_{1}\}$.
\end{enumerate}

\noindent Then $S$ is an edge metric generator in $G$.
\end{lemma}

\begin{proof}
The proof is analogous to the proof of Lemma \ref{Lemma_dim2}. Let $a$ be a
vertex in $G$. By $\mathcal{P}_{a}^{e}$ we denote the partition of $E(G)$
according to the distances from $a$. To prove that $S=\{a,b\}$ is an edge
metric generator for $G$, it suffices to show that $d(b,e)\neq d(b,f)$ for
every pair of edges $e,f$ from a common set of $\mathcal{P}_{a}^{e}$. Also, to
abbreviate the notation, an edge $u_{i}u_{i+1}$ will be denoted by $u_{i}^{+}$
or $u_{i+1}^{-},$ and the similar notation will be used for edges
$v_{i}v_{i+1}$ and $w_{i}w_{i+1}$. We now consider each of the five cases separately.

\medskip i) Denote $a=(r-p)/2$ and $b=(r+p)/2$. Then for $w_{a}\in S$ we have
\[
\mathcal{P}_{w_{a}}^{e}=(\{w_{a-i}^{-},w_{a+i}^{+}\}_{i=0}^{a-1},\{u_{i}%
^{+},v_{i}^{+},w_{r-p+i}^{+}\}_{i=0}^{p-1},\{v_{p+i}^{+},v_{q-i}^{-}%
\}_{i=0}^{\lfloor\frac{q-p}{2}\rfloor}).
\]
In the next we suppose that $w_{b}$ has the same distance to a pair of edges
from a common set of $\mathcal{P}_{w_{a}}^{e}$ and we always come to a
contradiction. Here and in the next cases, the first distance is denoted by
$d_{1}$ and the second distance is denoted by $d_{2}$.

Let us first consider the set $\{w_{a-i}^{-},w_{a+i}^{+}\}$ from
$\mathcal{P}_{w_{a}}^{e}.$ If $d(w_{b},w_{a-i}^{-})=d(w_{b},w_{a+i}^{+})$ then
$d_{1}=d(w_{b},w_{a-i})$. Further $d_{2}=d(w_{b},w_{a+i})$ (otherwise
$d_{2}<d_{1}$) and so $d(w_{b},w_{a-i})=d(w_{b},w_{a+i})$. Consequently
$b-(a-i)=(a+i)-b$ and therefore $a=b$ which contradicts $p\geq1$.

Let us now consider the set $\{u_{i}^{+},v_{i}^{+},w_{r-p+i}^{+}%
\}\in\mathcal{P}_{w_{a}}^{e}$ and the distances from $w_{b}$ to the three
possible pairs of edges from this set. If $d(w_{b},u_{i}^{+})=d(w_{b}%
,v_{i}^{+})$, then $d_{1}=d(w_{b},u_{i+1})$ since $d_{1}<d(w_{a}%
,v)=d(w_{b},u)$. Analogously $d_{2}=d(w_{b},v_{i+1})$. Thus $p-(i+1)=q-({i+1)}%
$ and $p=q$, a contradiction. The next pair is $u_{i}^{+}$ and $w_{r-p+i}%
^{+},$ where assuming $d(w_{b},u_{i}^{+})=d(w_{b},w_{r-p+i}^{+})$ yields
$d_{1}=d(w_{b},u_{i+1})$ and $d_{2}=d(w_{b},w_{r-p+i+1})$. Thus
\[
r-\frac{r+p}{2}+p-(i+1)=\frac{r+p}{2}-(r-p+i+1)
\]
which reduces to $r=p$, a contradiction. The last pair is $v_{i}^{+}$ and
$w_{r-p+i}^{+},$ in which case $d(w_{b},v_{i}^{+})=d(w_{b},w_{r-p+i}^{+})$
implies $d_{2}>d(w_{b},v)=d(w_{a},u)$, and so $d_{2}=d(w_{b},w_{r-p+i+1})$ and
$d_{1}=d(w_{b},v_{i+1})$. This gives
\[
r-\frac{r+p}{2}+q-(i+1)=\frac{r+p}{2}-(r-p+i+1)
\]
and $r+q=2p$, a contradiction.

It remains to consider the set $\{v_{p+i}^{+},v_{q-i}^{-}\}\in\mathcal{P}%
_{w_{a}}^{e}$. Assuming $d(w_{b},v_{p+i}^{+})=d(w_{b},v_{q-i}^{-})$ yields
$d_{2}=d(w_{b},v_{q-i})=d(w_{b},v)+d(v,v_{q-i})$. However,
$d(\{u,v\},\{v_{p+i},v_{p+i+1}\})>d(v,v_{q-i})$ and $d(\{u,v\},w_{b})\geq
d(w_{b},v)$. So $d_{1}>d_{2}$, a contradiction.

\medskip ii) Since $r\geq p+3$, we have $\lfloor(r-p)/2\rfloor\geq
\lfloor3/2\rfloor=1$. And since $p\geq1$, we have $\lfloor(r-p)/2\rfloor
<\lceil(r+p)/2\rceil$. Hence $1\leq\lfloor(r-p)/2\rfloor<\lfloor r/2\rfloor$.
Denote $a=\lfloor(r-p)/2\rfloor$ and $b=\lceil(r+p)/2\rceil$. Then for
$w_{a}\in S$ we have
\[
\mathcal{P}_{w_{a}}^{e}=(\{w_{a-i}^{-},w_{a+i}^{+}\}_{i=0}^{a-1},\{u_{i}%
^{+},v_{i}^{+},w_{2a+i}^{+}\}_{i=0}^{p-1},\{v_{p}^{+},v_{q}^{-},w_{r-1}%
^{+}\},\{v_{p+i}^{+},v_{q-i}^{-}\}_{i=1}^{\lfloor\frac{q-p-1}{2}\rfloor}).
\]
For each of the sets from $\mathcal{P}_{w_{a}}^{e}$, we now show that all
possible pairs of edges from that set are distinguished by $w_{b}\in S$. Let
us first consider the set $\{w_{a-i}^{-},w_{a+i}^{+}\}$. Assuming
$d(w_{b},w_{a-i}^{-})=d(w_{b},w_{a+i}^{+})$, analogously as in i) we get $a=b$
which contradicts $p\geq1$.

Now consider $\{u_{i}^{+},v_{i}^{+},w_{2a+i}^{+}\}$. If $d(w_{b},u_{i}%
^{+})=d(w_{b},v_{i}^{+})$, then analogously as in i) we get $p=q$, a
contradiction. If $d(w_{b},u_{i}^{+})=d(w_{b},w_{2a+i}^{+})$, then
$d_{1}=d(w_{b},u_{i+1})$ and $d_{2}=d(w_{b},w_{2a+i+1})$. Thus
\[
r-\frac{r+p+1}{2}+p-(i+1)=\frac{r+p+1}{2}-(2\frac{r-p-1}{2}+i+1)
\]
which reduces to $r=p+2$, a contradiction. Finally, if $d(w_{b},v_{i}%
^{+})=d(w_{b},w_{2a+i}^{+})$, then $d_{1}=d(w_{b},v_{i+1})$ and $d_{2}%
=d(w_{b},w_{2a+i+1})$. Thus
\[
r-\frac{r+p+1}{2}+q-(i+1)=\frac{r+p+1}{2}-(2\frac{r-p-1}{2}+i+1)
\]
and hence $r+q=2p+2$, a contradiction.

For edges from $\{v_{p}^{+},v_{q}^{-},w_{r-1}^{+}\}$ we have $d(w_{b}%
,w_{r-1}^{+})=d(w_{b},w_{r-1})=r-b-1$ and $d(w_{b},v_{q}^{-})=d(w_{b},v)=r-b$,
so that $d(w_{b},w_{r-1}^{+})<d(w_{b},v_{q}^{-})$. If $d(w_{b},w_{r-1}%
^{+})=d(w_{b},v_{p}^{+})$ then $d_{2}=d(w_{b},v_{p})$. So $r-b-1=b+p$ and
consequently $r-p-1=2b=r+p+1$, a contradiction. Finally, if $d(w_{b},v_{q}%
^{-})=d(w_{b},v_{p}^{+})$ then $d_{2}=d(w_{b},v_{p})$. So $r-b=b+p$ and
consequently $r-p=2b=r+p+1$, a contradiction.

Finally, consider $\{v_{p+i}^{+},v_{q-i}^{-}\}$. Assuming $d(w_{b},v_{p+i}%
^{+})=d(w_{b},v_{q-i}^{-})$ yields
\begin{align*}
d_{2}  &  =d(w_{b},v_{q-i})=d(w_{b},v)+d(v,v_{q-i})\\
&  <\min\{d(\{u,v\},\{v_{p+i},v_{p+i+1}\})+d(w_{b},\{u,v\})\}\leq d_{1},
\end{align*}
a contradiction.

\medskip iii) Notice that in this case $G=\Theta_{p,p+1,p+1}$, where $p\geq2$.
For $v_{1}\in S$ the partition is
\[
\mathcal{P}_{v_{1}}^{e}=(\{v_{1}^{-},v_{1}^{+}\},\{u_{i}^{+},v_{i+2}^{+}%
,w_{i}^{+}\}_{i=0}^{p-2},\{u_{p-1}^{+},w_{p-1}^{+},w_{p+1}^{-}\}).
\]
First consider the set $\{v_{1}^{-},v_{1}^{+}\}$. Since $p\geq2$, we have
$d(w_{1},v_{1}^{-})=1<2=d(w_{1},v_{1}^{+})$, so $v_{1}^{-}$ and $v_{1}^{+}$
are distinguished by $w_{1}\in S$.

Now consider $\{u_{i}^{+},v_{i+2}^{+},w_{i}^{+}\}$. Since
\[
d(w_{1},w_{i}^{+})\leq i<d(w_{1},u_{i}^{+})=i+1<i+2\leq d(w_{1},v_{i+2}^{+}),
\]
all three pairs are distinguished by $w_{1}\in S.$

Finally, for $\{u_{p-1}^{+},w_{p-1}^{+},w_{p+1}^{-}\}$ we have%
\[
d(w_{1},w_{p-1}^{+})=p-2<d(w_{1},w_{p+1}^{-})=p-1<d(w_{1},u_{p-1}^{+})=p.
\]

\medskip iv) Observe that
\[
\mathcal{P}_{v_{1}}^{e}=(\{v_{1}^{-},v_{1}^{+}\},\{u_{i}^{+},v_{i+2}^{+}%
,w_{i}^{+}\}_{i=0}^{p-3},\{u_{p-2}^{+},u_{p}^{-},w_{p-2}^{+},w_{r}%
^{-}\},\{w_{p-1+i}^{+},w_{r-1-i}^{-}\}_{i=0}^{\lceil\frac{r-p}{2}\rceil}).
\]
First, for the unique pair from $\{v_{1}^{-},v_{1}^{+}\}$ it holds that
$d(u_{2},v_{1}^{-})=2<d(u_{2},v_{1}^{+})=3$ if $p\geq4$, so it is
distinguished by $u_{2}$.

Next, consider $\{u_{i}^{+},v_{i+2}^{+},w_{i}^{+}\}$. Suppose that
$d(u_{2},u_{i}^{+})=d(u_{2},v_{i+2}^{+})$. If $i\ge2$ then $d_{1}=i-2$ and
consequently $d_{2}=i+4$, a contradiction. Hence $0\le i\le2$ and
\[
d(u_{2},u_{i}^{+})=d(u_{2},u_{i+1})=1-i<d(u_{2},v_{i+2}^{+})=d(u_{2}%
,v_{i+3})=p-2+p-(i+3),
\]
a contradiction. For the second pair $u_{i}^{+}$ and $w_{i}^{+}$, since $i\leq
p-3$ we have
\[
d(u_{2},u_{i}^{+})\leq(i-2)+3<i+2=d(u_{2},w_{i})=d(u_{2},w_{i}^{+}).
\]
For the last pair $v_{i+2}^{+}$ and $w_{i}^{+},$ we assume that $d(u_{2}%
,v_{i+2}^{+})=d(u_{2},w_{i}^{+})$. We distinguish three subcases:

\begin{enumerate}
\item[-] if $0\leq i\leq p-5$ then $d_{1}=d(u_{2},v_{i+2})=i+4>i+2=d(u_{2}%
,w_{i})=d_{2}$;

\item[-] if $i=p-4$ then $d_{1}=d(u_{2},v_{i+3})=p-1>p-2=d(u_{2},w_{i})=d_{2}$;

\item[-] if $i=p-3$ then $d_{1}=d(u_{2},v_{i+3})=p-2<p-1=d(u_{2},w_{i})=d_{2}$.
\end{enumerate}

Now we consider the set $\{u_{p-2}^{+},u_{p}^{-},w_{p-2}^{+},w_{r}^{-}\}$. We
have
\[
d(u_{2},u_{p-2}^{+})=p-4<d(u_{2},u_{p}^{-})=p-3< d(u_{2},w_{r}^{-}%
)=p-2<d(u_{2},w_{p-2}^{+})\geq p-1,
\]
where the last inequality is an equality only if $r=p$.

Finally, for $\{w_{p-1+i}^{+},w_{r-1-i}^{-}\}$ suppose that $d(u_{2}%
,w_{p-1+i}^{+})=d(u_{2},w_{r-1-i}^{-})$. Then $d_{2}=d(u_{2},w_{r-1-i})$ and
so $d_{1}=d(u_{2},w_{p-1+i})$. Thus $2+p-1+i=p-2+r-(r-1-i)$ and $1=-1$, a contradiction.

\medskip v) Observe that
\[
\mathcal{P}_{v_{1}}^{e}=(\{v_{1}^{-},v_{1}^{+}\},\{u_{i}^{+},v_{i+2}^{+}%
,w_{i}^{+}\}_{i=0}^{p-3},\{u_{p-2}^{+},u_{p}^{-},w_{p-2}^{+},w_{r}%
^{-}\},\{w_{p-1+i}^{+},w_{r-1-i}^{-}\}_{i=0}^{\lceil\frac{r-p}{2}\rceil}).
\]
First, consider the set $\{v_{1}^{-},v_{1}^{+}\}$. Since $r\geq4$, we have
$d(w_{1},v_{1}^{-})=1<2=d(w_{1},v_{1}^{+})$.

Next, consider the set $\{u_{i}^{+},v_{i+2}^{+},w_{i}^{+}\}$. We have
\[
d(w_{1},v_{i+2}^{+})=i+3>d(w_{1},u_{i}^{+})=i+1>d(w_{1},w_{i}^{+})\leq i,
\]
where the last inequality is an equality only if $i=0$ and the first equality
is a consequence of $r\geq p+3$.

Now consider the set $\{u_{p-2}^{+},u_{p}^{-},w_{p-2}^{+},w_{r}^{-}\}$. We
have
\[
d(w_{1},w_{p-2}^{+})=p-3<d(w_{1},u_{p-2}^{+})=p-1< d(w_{1},u_{p}%
^{-})=p<d(w_{1},w_{r}^{-})=p+1,
\]
where the last inequality holds since $r\geq p+3$.

Finally, consider the set $\{w_{p-1+i}^{+},w_{r-1-i}^{-}\}$. If $d(w_{1}%
,w_{p-1+i}^{+})=d(w_{1},w_{r-1-i}^{-})$, then $d_{1}=d(w_{1},w_{p-1+i})$ and
so $d_{2}=d(w_{1},w_{r-1-i})$. Thus $p-1+i-1=1+p+r-(r-1-i)$ and $-2=2$, a
contradiction. This concludes the proof.
\end{proof}

\begin{figure}[h]
\begin{center}
$%
\begin{array}
[c]{lll}%
\text{i) \raisebox{-1\height}{\includegraphics[scale=0.6]{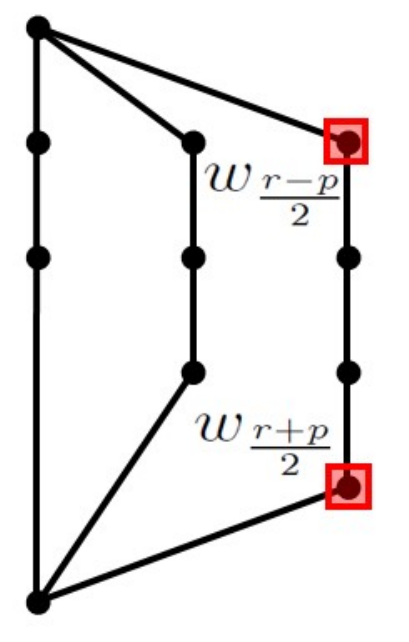}}} &
\text{ii) \raisebox{-1\height}{\includegraphics[scale=0.6]{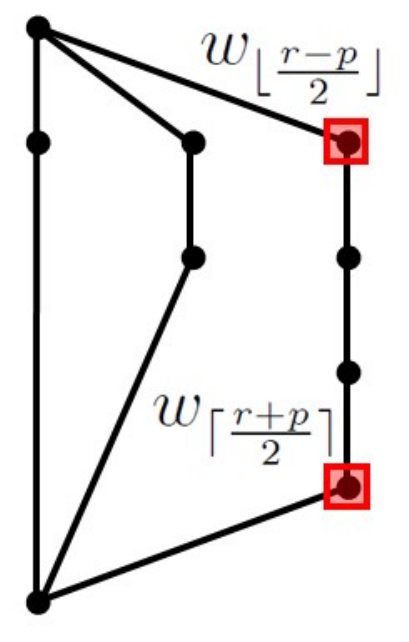}}} &
\text{iii) \raisebox{-1\height}{\includegraphics[scale=0.6]{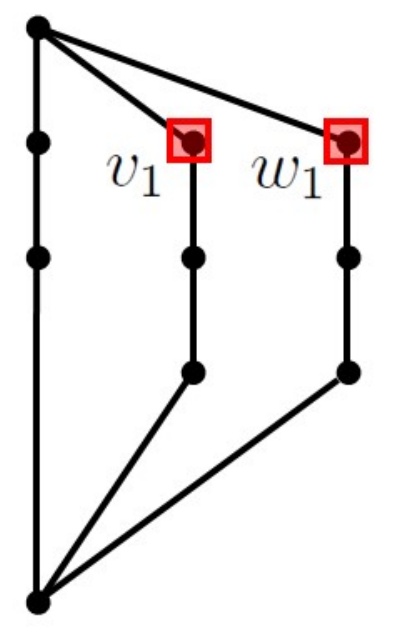}}}\\
\text{iv) \raisebox{-1\height}{\includegraphics[scale=0.6]{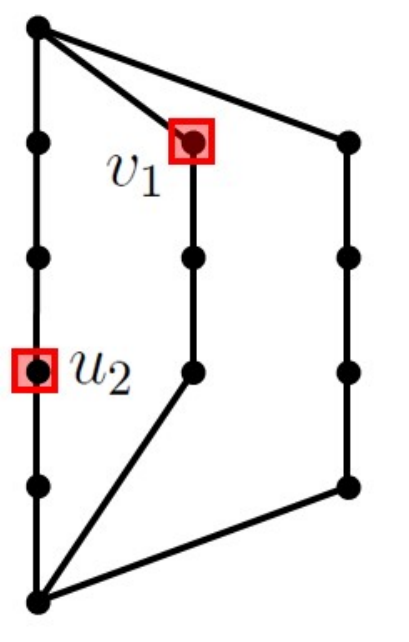}}} &
\text{v) \raisebox{-1\height}{\includegraphics[scale=0.6]{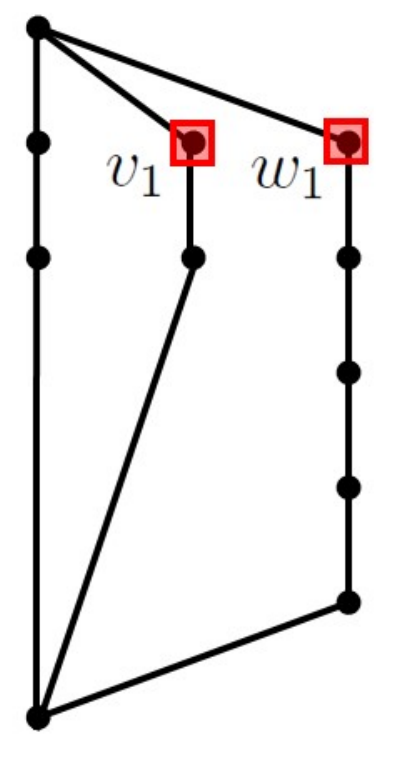}}} &
\end{array}
$
\end{center}
\caption{Edge metric generators from Lemma \ref{Lemma_edim2}.}%
\label{Fig_EdgeGenerators2}%
\end{figure}

The following statement is a consequence of Lemma~{\ref{Lemma_edim2}}.

\begin{theorem}
\label{thm:edim=2} Let $G$ be a $\Theta$-graph such that $G\not =%
\Theta_{1,2,2}$ and $\Theta_{p,p,q}$ for $2\leq p\leq3$ and $p\leq q\leq p+2$.
Then $\mathrm{edim}(G)=2$.
\end{theorem}

\begin{proof}
First suppose that $p<q$. If $p+r$ is even then $r\ge3$, so
Lemma~{\ref{Lemma_edim2}}.i) covers this case. On the other hand if $p+r$ is
odd then $r\ge p+1$, so Lemma~{\ref{Lemma_edim2}}.ii) and
Lemma~{\ref{Lemma_edim2}}.iii) cover all cases except $\Theta_{1,2,2}$.

Now suppose that $p=q$. Then Lemma~{\ref{Lemma_edim2}}.v) covers all cases
except $\Theta_{p,p,p}$, $\Theta_{p,p,p+1}$ and $\Theta_{p,p,p+2}$. These
remaining cases are covered by Lemma~{\ref{Lemma_edim2}}.iv) when $p\ge4$.
Hence, uncovered cases are $\Theta_{2,2,2}$, $\Theta_{2,2,3}$, $\Theta
_{2,2,4}$, $\Theta_{3,3,3}$, $\Theta_{3,3,4}$ and $\Theta_{3,3,5}$.
\end{proof}

Supporting our motivation, Theorem~{\ref{thm:edim=2}} yields the following corollary.

\begin{corollary}
Let $G$ be a $\Theta$-graph such that $G\not =\Theta_{1,2,2}$ and
$\Theta_{p,p,q}$ for $2\leq p\leq3$ and $p\leq q\leq p+2$. Then $\mathrm{edim}%
(G)<2c(G)-1.$
\end{corollary}

\section{Further work}

In this paper we investigated Conjecture \ref{Con_dim_leaves} (resp.
Conjecture \ref{Con_edim_leaves}), which states that the vertex (resp. the
edge) metric dimension of a graph $G\not =C_{n}$ with $\delta(G)\geq2$ is
bounded above by $2c(G)-1$. It was established in \cite{SedSkreLeaflessCacti}
that the conjectures hold for cacti without leaves, and that for other
leafless graphs the problem reduces to $2$-connected graphs, i.e., if the
conjectures hold for $2$-connected graps distinct from a cycle then they hold
in general. In this paper we considered $\Theta$-graphs, since they are the
most simple $2$-connected graps distinct from cycles. We established that
Conjectures \ref{Con_dim_leaves} and \ref{Con_edim_leaves} hold on this class
of graphs and we characterized all $\Theta$-graphs for which the upper bound
is attained.

Besides $\Theta$-graphs attaining the upper bound $2c(G)-1$, it was previously
established that the same upper bound is also attained by metric dimensions of
some leafless cacti. To be more precise, a \emph{daisy} graph is any graph
consisting of at least two cycles which all share the same vertex. A cycle in
a daisy graph is also called a \emph{petal}. Now, it was established that
$\mathrm{dim}(G)$ attains the bound $2c(G)-1$ if $G$ is a daisy graph without
odd petals, and that $\mathrm{edim}(G)$ reaches the same bound for any daisy
graph $G$. We expect that these graphs are the only graphs with $\delta
(G)\geq2$ whose metric dimensions reach the bound. So we conclude the paper by
stating the following two conjectures.

\begin{conjecture}
\label{Cor_dim_equality}Let $G$ be a connected graph with $\delta(G)\geq2$.
Then $\mathrm{dim}(G)=2c(G)-1$ if and only if $G$ is a daisy graph without odd
petals, $G=\Theta_{p,p,p}$ or $G=\Theta_{p,p,p+2}.$
\end{conjecture}

\begin{conjecture}
\label{Cor_edim_equality}Let $G$ be a connected graph with $\delta(G)\geq2$.
Then $\mathrm{edim}(G)=2c(G)-1$ if and only if $G$ is a daisy graph,
$G=\Theta_{1,2,2}$ or $G=\Theta_{p,p,q}$ with $2\leq p\leq3$ and $p\leq q\leq
p+2.$
\end{conjecture}

Similarly as with Conjectures \ref{Con_dim_leaves} and \ref{Con_edim_leaves},
we show in the next proposition that the above two conjectures reduce to the
same problem on $2$-connected graphs. In order to do so we will use a result
from \cite{SedSkreLeaflessCacti}, which states that $c(G)=c(G_{1}%
)+\cdots+c(G_{q})$ where $G_{1},\ldots,G_{q}$ is the complete list of blocks
of $G$.

\begin{proposition}
If Conjecture \ref{Cor_dim_equality} (resp. Conjecture \ref{Cor_edim_equality}%
) holds for $2$-connected graphs, then it holds in general.
\end{proposition}

\begin{proof}
We say that $G$ is \emph{vertex extremal}, if $G=\Theta_{p,p,p}$ or
$G=\Theta_{p,p,p+2}.$ We say $G$ is \emph{edge extremal} if $G=\Theta_{1,2,2}$
or $G=\Theta_{p,p,q}$ for $2\leq p\leq3$ and $p\leq q\leq p+2.$ Now, let $G$
be a graph with $\delta(G)\geq2$ which is not $2$-connected. According to
Lemma \ref{Cor_attained}, the equality $\mathrm{dim}(G)=2c(G)-1$ (resp.
$\mathrm{edim}(G)=2c(G)-1$) may hold only when every non-trivial block of $G$
distinct from a cycle is vertex extremal (resp. edge extremal) and all blocks
of $G$ share a vertex.

We shall now construct a vertex (resp. an edge) metric generator in such a
graph whose size is smaller than $2c(G)-1,$ which is sufficient to prove the
claim. Let $v$ be a vertex of $G$ shared by all blocks in $G.$ Let us assume
$G_{1},\ldots,G_{q}$ are all non-trivial blocks in $G$ denoted so that $G_{i}$
is a cycle whenever $i>p.$ According to Lemma \ref{Lemma_vertexGenerators}
(resp. Lemma \ref{Lemma_edgeGenerators}), for $1\leq i\leq p$ there is a
vertex (resp. an edge) metric generator $S_{i}^{\prime}$ in $G_{i}$ such that
$v\in S_{i}^{\prime},$ and for such $i$ let us denote $S_{i}=S_{i}^{\prime
}\backslash\{v\}$. For $i>p,$ let $S_{i}$ consist of a single vertex which is
a neighbor of $v$ in $G_{i}.$ Now, let $S=S_{1}\cup\cdots\cup S_{q}.$ Observe
that the set $S$ distinguishes in $G$ all pairs of vertices (resp. edges)
which belong to the same block of $G$, this follows from the fact that a pair
of vertices (resp. edges) which is distinguished by $v$ in $G_{i}$ is in $G$
distinguished by every vertex $s\in S\backslash V(G_{i}).$

By above, a pair of vertices (resp. edges) $x$ and $x^{\prime}$ is not
distinguished by $S$ in $G$ only if $x$ belongs to $G_{i}$ and $x^{\prime}$
belongs to $G_{j}$, $i\neq j$. In such a case we say $G_{i}$ and $G_{j}$ are
\emph{critically incident}. So let $G_{i}$ and $G_{j}$ are critically incident
with $x\in V(G_{i})$, $x^{\prime}\in V(G_{j})$ such that $x$ and $x^{\prime}$
are not distinguished by $S$. Let further $s\in S_{i}$ and $s^{\prime}\in
S_{j}$. Then $d(s,x)=d(s,x^{\prime})$ and $d(s^{\prime},x)=d(s^{\prime
},x^{\prime})$. Denote $a=d(s,v)$, $b=d(v,x)$, $c=d(s^{\prime},v)$ and
$d=d(v,x^{\prime})$. Then
\begin{align*}
(a+d)+(c+b)  &  =d(s,x^{\prime})+d(s^{\prime},x)=d(s,x)+d(s^{\prime}%
,x^{\prime})\\
&  \leq d(s,v)+d(v,x)+d(s^{\prime},v)+d(v,x^{\prime})=a+b+c+d,
\end{align*}
and so a shortest path from $x$ (resp. $x^{\prime}$) to every vertex from
$S_{i}$ (resp. $S_{j}$) leads through $v$. Hence $b=d$ and $a=c$.

If $b>1$ then let $x_{1}$ (resp. $x_{1}^{\prime}$) be a neighbor of $x$ (resp.
$x^{\prime}$) on a shortest path from $v$ to $x$ (resp. $x^{\prime}$). Then
for every $s^{*}\in S_{i}\cup S_{j}$ we have $d(s^{*},x_{1})=d(s^{*}%
,x_{1}^{\prime})=a+b-1$, so $x_{1}$ and $x_{1}^{\prime}$ are not distinguished
by $S$ as well.

Finally, let $x_{2}$ and $x_{2}^{\prime}$ be another pair of vertices which is
not distinguished by $S$, $x_{2}\in V(G_{i})$ and $x_{2}^{\prime}\in V(G_{j}%
)$, and let $d(v,x_{2})=d(v,x)$. Then $x$ and $x_{2}$ are not distinguished by
$S_{i}$, which means that $x_{2}=x$ and analogously $x_{2}^{\prime}=x^{\prime
}$.

Thus vertices $y\in V(G_{i})$, for which there exists $y^{\prime}\in V(G_{j})$
such that $y,y^{\prime}$ is a pair not distinguished by $S$, form a path
starting at a neighbor of $v$. Denote this neighbor by $z$. If there is $k\ne
j$ such that $G_{i}$ and $G_{k}$ are critically incident as well, then again
vertices $y\in V(G_{i})$, for which there exists $y^{*}\in V(G_{k})$ such that
$y,y^{*}$ is a pair not distinguished by $S$, form a path starting at $z$. So
it is sufficient to add $z$ to $S_{i}$ and all pairs of vertices from $G_{i}$
and $G_{k}$ (as well as from $G_{i}$ and $G_{j}$) will be distinguished.

We conclude that it is sufficient to introduce to $S$ at most $q-1$ vertices,
and all pairs $x$ and $x^{\prime}$ from distinct blocks will also be
distinguished by $S.$ Consequently, since $\left\vert S_{i}\right\vert
=\mathrm{dim}(G_{i})-1$ we have%
\begin{align*}
\mathrm{dim}(G)  &  \leq\sum_{i=1}^{q}(\mathrm{dim}(G_{i})-1)+q-1=\sum
_{i=1}^{q}\mathrm{dim}(G_{i})-1\\
&  =\sum_{i=1}^{p}(2c(G_{i})-1)+\sum_{i=p+1}^{q}2c(G_{i})-1=2c(G)-p-1
\end{align*}
which is obviously smaller than $2c(G)-1$ for $p\geq1.$ If $p=0,$ then $G$ is
a cactus graph and for cacti it was already established that the bound is
attained only for daisy graphs without odd petals. The proof for
$\mathrm{edim}(G)$ is analogous.
\end{proof}

\bigskip

\bigskip\noindent\textbf{Acknowledgments.}~~The first author aknowledges
partial support by Slovak research grants VEGA 1/0567/22, VEGA 1/0206/20,
APVV-19-0308, APVV--17--0428. The second author acknowledges the support of
Project KK.01.1.1.02.0027, a project co-financed by the Croatian Government
and the European Union through the European Regional Development Fund - the
Competitiveness and Cohesion Operational Programme. All authors acknowledge
partial support of the Slovenian research agency ARRS program\ P1-0383 and
ARRS project J1-1692.

\end{document}